\documentclass{amsart}

\numberwithin{equation}{section}
\usepackage{graphicx}
\usepackage{amssymb}
\usepackage{amsmath}
\usepackage{setspace}
\usepackage{bbm}
\usepackage{color}

\allowdisplaybreaks

\usepackage{hyperref}
\usepackage{cleveref}

\newcommand{\Aut}{\mathrm{Aut}}

\newcommand{\mcal}{\mathcal}

\newcommand{\mbb}{\mathbb}

\newcommand{\mA}{\mathcal{A}}
\newcommand{\mB}{\mathcal{B}}
\newcommand{\mC}{\mathcal{C}}
\newcommand{\mD}{\mathcal{D}}
\newcommand{\mP}{\mathcal{P}}

\newcommand{\omin}{\otimes^{\min}}

\newtheorem{theorem}{Theorem}[section]
\newtheorem{lemma}[theorem]{Lemma}
\newtheorem{cor}[theorem]{Corollary}
\newtheorem{definition}[theorem]{Definition}
\newtheorem{proposition}[theorem]{Proposition}
\newtheorem{remark}[theorem]{Remark}

\begin{document}
\title[{\tiny Distances between subalgebras
   of operator algebras}]{
   On stability of distance under some tensor products and some calculations
}

 \author[S Kumar]{Sumit Kumar} \address{School of Physical Sciences, Jawaharlal Nehru University, New Delhi, India}
 \email{sumitkumar.sk809@gmail.com}



\thanks{The author was  supported
  by the Council of Scientific and Industrial Research (Government
  of India) through a Senior Research Fellowship with
  CSIR File No.{\bf 09/0263(12012)/2021-EMR-I}}
\maketitle

\begin{abstract}
We prove that the Kadison-Kastler and Christensen distances are stable
under the Banach space injective tensor product (resp., the
Banach space projective tensor product) of a Banach space with any
unital commutative $C^*$-algebra (resp., of a $C^*$-algebra with any
unital $C^*$-algebra).

Apart from these stability results, we make some explicit calculations
of the Kadison-Kastler, Christensen and Mashood-Taylor distances
between certain subalgebras of some crossed-product operator algebras.
  \end{abstract}
\section{Introduction}
A notion of distance between subalgebras of $B(\mcal{H})$ was first
introduced by Kadison and Kastler in 1973 in \cite{KK}.  This notion
gathered much attention in the 80s and 90s, particularly in the
context of perturbation theory for operator algebras.  Notable
contributions in this area were made by Christensen, Phillips and
Johnson, who employed the Kadison–Kastler distance to establish
several significant results concerning the stability of $C^*$-algebras
and von Neumann algebras under small perturbations.

Recently, Ino and Watatani, in \cite{IW}, used this notion in the
study of inclusions of $C^*$-algebras. Motivated by their work, a
similar perspective was adopted in an earlier work \cite{GK}, wherein
this notion was applied effectively to study the lattice of
intermediate subalgebras of inclusions of $C^*$-algebras and some
concrete calculations were made. Motivated by the calculations made in
\cite{GK} for crossed product $C^*$-algebras, we also initiated the study of
such distances between subalgebras of the spatial tensor product of
$C^*$-algebras, in \cite{GK2}.

An important open problem in the theory of operator algebras concerns
the behaviour of the Kadison–Kastler distance under tensor
products. Specifically, as already mentioned by Christensen et al in
\cite{Ch3}, it is not yet fully understood how the distance between
two $C^*$-algebras relates to the distance between their respective
tensor products with a fixed (nuclear) $C^*$-algebra. That is, given
two $C^*$-algebras $\mathcal{A}$ and $\mathcal{B}$ on some Hilbert
space $\mcal{H}$ with small Kadison–Kastler distance, it remains an
open question whether the tensor products $\mathcal{A} \otimes^{\min}
\mathcal{C}$ and $\mathcal{B} \otimes^{\min} \mathcal{C}$ also remain
close (as subalgebras of $B(\mcal{H})\otimes^{\min} \mcal{C}$) in the
Kadison–Kastler sense, for a fixed nuclear $C^*$-algebra
$\mathcal{C}$. Similar question can be asked for the Christensen
distance as well. In this direction, various attempts were made by
Christensen et al (see, for instance, \cite{Ch1}, \cite{Ch2} and
\cite{Ch3}). (As in \cite{GK, GK2}, we denote the Kadison-Kastler and
Christensen distances by $d_{KK}$ and $d_0$, respectively.) Here are
two of Christensen's noteworthy results in this direction:

\begin{theorem}\cite[Theorem 3.1]{Ch1}\label{Ch1-6kd}
Let $\mC$ be a $C^*$-algebra with $C^*$-subalgebras $\mA$ and $\mB$,
and let $\mD$ be a nuclear $C^*$-algebra. If $\mA$ has property
$D_{k}$ (for some $k\in (0, \infty))$ and $\mA \subseteq_{\gamma}
\mB$, then $\mA \otimes^{\min} \mD \subseteq_{6k\gamma} \mB\otimes^{\min}
\mD$.

In particular, when $\mA$ and $\mB$ both have property $D_{k}$, then
\[
d_{0}(\mA \otimes^{\min} \mD, \mB \otimes^{\min} \mD)\leq 6k d_{0}(\mA,\mB).
\]
\end{theorem}

\begin{theorem}\cite[Corollary 4.7]{Ch3}
Let $\mC$ be a $C^*$-algebra with $C^*$-subalgebras $\mA$ and $\mB$,
and let $\mD$ be a nuclear $C^*$-algebra. Then, for every $l, K\in
\mbb{N}$, there exists a constant $L_{l,K}$ (depending only on $l$ and
$K$) such that when $\mA$ has length at most $l$ with length constant
at most $K$, then
\[
d_{KK}(\mA \otimes^{\min} \mD, \mB \otimes^{\min} \mD)\leq L_{l,K}d_{KK}(\mA, \mB).
\] 

\end{theorem}

Very recently, using some techniques from \cite{Ch1}, the following
 stability result in the category of $C^*$-algebras was achieved in \cite{GK2}:

\begin{theorem}\cite[Theorem 4.6]{GK2}
Let $\mD$ be a commutative $C^{*}$-algebra and, 
$\mA$ and $\mB$ be
  $C^*$-subalgebras of a $C^*$-algebra $\mC$. Then,
\[
d_{KK}(\mA\omin \mD, \mB\omin \mD)\leq d_{KK}(\mA, \mB).
\]
Moreover, if $\mD$ is unital, then
\[
d_{KK}(\mA\omin \mD, \mB\omin \mD)= d_{KK}(\mA, \mB).
\]
\end{theorem}
It is noteworthy that such comparisons were used effectively by
Christensen and others in proving some fundamental results related to
perturbation and preserving of certain invariants of close
subalgebras. For instance:
\begin{enumerate}
  \item[(a)] In \cite[Theorem 4.3]{Ch1}, Christensen
proved that if an injective von Neumann algebra $N$ admits a near
inclusion in some von Neumann algebra $M$, then $N$ embeds in $M$ via a
unitary conjugation.
\item[ (b)] In \cite{Khos84}, Khoshkam proved that
if the matrix-amplifications of two $C^*$-algebras are sufficiently
(uniformly) close, then they have isomorphic $K$-groups; which, by one
of the above mentioned results of Christensen (see \Cref{Ch1-6kd})
allows one to deduce that if $\mA$ is a nuclear $C^*$-algebra, then
any sufficiently close $C^*$-algebra to $\mA$ has 
$K$-groups isomorphic to  those of $\mA$.
\item[(c)] Moreover, Christensen et al (\cite[$\S 5$]{Ch3}) exploited
  the techniques of Khoshkam and Christensen further to prove that
  $C^*$-algebras  sufficiently close to a $C^*$-algebra belonging to  a certain family have
  isomorphic (ordered) $K$-theories.
\end{enumerate}

Over the last four decades or so, the study of non-operator algebraic
tensor products (like the Banach space and operator space projective
tensor products and the Haagerup tensor product) of $C^*$-algebras
have become very relevant in  the world of operator
algebras - see, for instance, \cite{Ble, ASS, BKS, KJ, GR} and the
references cited therein. It is thus natural to ask similar questions (as above) from
the perspective of non-operator algebraic tensor products as well. In
this direction, we obtain some analogues of the above-mentioned results with
respect to the injective and projective Banach space tensor
products. More precisely, we prove the following:\medskip

\noindent {\em {\bf \Cref{equality}.}  Let $\mcal{X}$ be a Banach
  space and $C(K)$ denote the Banach space consisting of continuous
  functions on a compact Hausdorff space $K$. Then, for any two closed subspaces
  $\mcal{Y}$ and $\mcal{Z}$ of $\mcal{X}$,
\[
d_{KK}(\mcal{Y}, \mcal{Z})= d_{KK}(\mcal{Y}\otimes^{\varepsilon} C(K),
\mcal{Z}\otimes^{\varepsilon} C(K))
\]
and      
\[
d_{0}(\mcal{Y}, \mcal{Z})= d_{0}(\mcal{Y}\otimes^{\varepsilon} C(K),
\mcal{Z}\otimes^{\varepsilon} C(K)).
\]}
\medskip

\noindent{\em {\bf \Cref{Projective-Ch}.}  Let $\mC$ and $\mD$ be two
  $C^*$-algebras and, $\mA$ and $ \mB$ be $C^*$-subalgebras of
  $\mC$. If $\beta$ is a positive scalar such that $\mA
  \subseteq_{\beta}\mB$, then $\mA \otimes^{\gamma} \mD
  \subseteq_{\beta}\mB \otimes^{\gamma} \mD$.  In particular,
$$d_0(\mA \otimes^{\gamma} \mD, \mB \otimes^{\gamma} \mD)  \leq d_0(\mA, \mB).$$
Moreover, if $\mD$ is unital, then
$$d_{0}(\mA \otimes^{\gamma} \mD, \mB \otimes^{\gamma} \mD) = d_{0}(\mA, \mB).$$ 
}

\noindent{\em {\bf \Cref{Projective-KK}.}  Let $\mC$ and $\mD$ be two
  $C^*$-algebras and, $\mA$ and $ \mB$ be $C^*$-subalgebras of $\mC$. Then,
 $$d_{KK}(\mA \otimes^{\gamma} \mD, \mB \otimes^{\gamma} \mD)  \leq d_{KK}(\mA, \mB).$$
 Moreover, if $\mD$ is unital, then
$$d_{KK}(\mA \otimes^{\gamma} \mD, \mB \otimes^{\gamma} \mD) = d_{KK}(\mA, \mB).$$  } 

It is natural to ask whether it is possible to compute the precise distance for certain operator algebras. To address this question,
 in \cite{GK}, we calculated the precise distance in several specific examples of operator algebras. In a similar spirit,
 in \Cref{twisted} and \Cref{crossedv}, we make some concrete
calculations of the Kadison–Kastler and Christensen distances between
subalgebras (associated to subgroups of a discrete group) of reduced
twisted crossed-product $C^*$-algebras and of crossed-product von
Neumann algebras. Furthermore, towards the end, we calculate the
Mashood-Taylor distance between the crossed product von Neumann
subalgebras (associated to subgroups of a discrete group) of the
crossed product of a discrete group with  a tracial von
Neumann algebra.

\section{Preliminaries}\label{prelims}

\subsection{Two notions of distance between subspaces of normed spaces}
\subsubsection{Kadison-Kastler distance}
For any normed space $\mcal{X}$, its closed unit ball will be denoted
by $B_1(\mcal{X})$ and for any subset $\mcal{S}$ of $\mcal{X}$ and an
element $x \in \mcal{X}$, as is standard, the distance between $x$ and
$\mcal{S}$ is given by \( d(x, \mcal{S}) = \inf\{ \|x - s\|: s \in
\mcal{S}\}.  \)

For any two subspaces $\mcal{Y}$ and $\mcal{Z}$ of a normed space
$\mcal{X}$, recall (from \cite{KK}) that the Kadison-Kastler distance
between them (which we denote by $d_{KK}(\mcal{Y},\mcal{Z})$) is
defined as the Hausdorff distance between their closed unit balls,
i.e.,
\[
d_{KK}(\mcal{Y},\mcal{Z}) := \max\left\{\sup_{y \in B_1(\mcal{Y})} d(y, B_1(\mcal{Z})),
\sup_{z\in B_1(\mcal{Z})}d(z, B_1(\mcal{Y}))\right\}.
\]

\begin{remark}{\label{KK-facts}}
  Let $\mcal{X}$ be a normed space.  Then, the following facts are well known:
\begin{enumerate}
\item $d_{KK}(\mcal{Y},\mcal{Z}) \leq 1$ for all subspaces $\mcal{Y},
  \mcal{Z}$ of $\mcal{X}$.
 \item $d_{KK}(\mcal{Y},\mcal{Z}) = d_{KK}(\overline{\mcal{Y}},
   \mcal{Z})= d_{KK}(\overline{\mcal{Y}}, \overline{\mcal{Z}})$ for
   all subspaces $\mcal{Y},\mcal{Z}$ of $\mcal{X}$.

\end{enumerate}
\end{remark}

\subsubsection{Near inclusions and Christensen distance}

 Let $\mcal{X}$ be a normed space. Recall from \cite{Ch2} that, for
 any two subspaces $\mcal{Y}$ and $ \mcal{Z}$ of $\mcal{X}$ and a
 scalar $\gamma > 0$, $\mcal{Y} \subseteq_{\gamma} \mcal{Z}$ if for
 each $y \in B_1(\mcal{Y})$, there exists a $ z \in \mcal{Z}$ such
 that $\|y-z\| \leq \gamma$; and, the Christensen distance between
 $\mcal{Y}$ and $\mcal{Z}$ is defined by
   $$
    d_0(\mcal{Y},\mcal{Z}) = \inf\{\gamma > 0 \, : \, \mcal{Y} \subseteq_{\gamma}
    \mcal{Z} \text{ and } \mcal{Z} \subseteq_{\gamma} \mcal{Y}\}.
  $$  

\begin{remark}\label{d0-basics}
  Let $\mcal{X}$ be a normed space. Then, the following useful facts are
  well known:
      \begin{enumerate}
      \item $d_0(\mcal{Y}, \mcal{Z}) \leq 1$ for all subspaces $\mcal{Y}, \mcal{Z}$ of $\mcal{X}$.
\item $d_0(\mcal{Y},\mcal{Z}) = d_0(\overline{\mcal{Y}}, \mcal{Z})= d_0(\overline{\mcal{Y}},
  \overline{\mcal{Z}})$ for all subspaces $\mcal{Y}, \mcal{Z}$ of $\mcal{X}$.
  \item The distances $d_0$ and $d_{KK}$
    are ``equivalent" in the sense that
  \[
  d_0 (\mcal{Y},\mcal{Z}) \leq d_{KK}(\mcal{Y},\mcal{Z}) \leq 2 d_0(\mcal{Y},\mcal{Z})
  \]
  for all subspaces $\mcal{Y},\mcal{Z}$ of $\mcal{X}$. (\cite[Remark
    2.3]{Ch2}) \end{enumerate}
      \end{remark}

\subsubsection{Mashood-Taylor distance} 
Let $\mcal{M}$ be a von Neumann algebra with a faithful normal tracial
state $\tau$. Recall from \cite[Section 5]{GK} that, for any two subalgebras
$\mcal{P}$ and $ \mcal{Q}$ of  $\mcal{M}$, the Mashood-Taylor
distance between them  is given by
\[
d_{MT} (\mcal{P}, \mcal{Q}) = d_{H, \|\cdot\|_\tau } (\widehat{B_1(\mcal{P})}, \widehat{B_1(\mcal{Q})}),
\]
where $d_{H, \|\cdot\|_\tau}$ denotes the Hausdorff distance with
respect to the metric induced by the norm $\|\cdot\|_\tau$ and, for
$\mcal{S} \subseteq \mcal{M}$, $\widehat{\mcal{S}}:= \{ \widehat{x} :
x \in \mathcal{S}\} \subseteq L^2(\mcal{M}, \tau)$.

\begin{remark}\label{SOT}
Let $(\mathcal{M},\tau)$ be as above. Then, for any two unital $*$-subalgebras $\mcal{P}$ and $ \mcal{Q}$ of $\mcal{M}$, 
\begin{enumerate}
\item $
d_{MT}(\mcal{P},\mcal{Q} ) = d_{MT}(\mcal{P}, \overline{\mcal{Q}}^{S.O.T.}) =
d_{MT}(\overline{\mcal{P}}^{S.O.T.}, \overline{\mcal{Q}}^{S.O.T.}) = d_{MT}(\mcal{P}'',
\mcal{Q}'')$; and, \hfill (\cite[Proposition 5.5]{GK})

\item $ 
d_{MT}(\mcal{P},\mcal{Q} )\leq d_{KK}(\mcal{P},\mcal{Q} ) $.\hfill  (\cite[Lemma 5.6]{GK})
\end{enumerate}
\end{remark}

\section{Distance between subspaces of injective tensor product of certain Banach spaces}\label{injective}
\subsection{Injective tensor product}
\begin{definition}
Let $\mcal{X}$ and $\mcal{Y}$ be Banach spaces and $\mcal{X}\otimes \mcal{Y}$ denote their algebraic tensor product. For any element $u= \sum_{i=1}^n x_i\otimes y_i \in \mcal{X}\otimes \mcal{Y}$, its injective norm is given by 
$$
\|u\|_{\varepsilon} = \sup \left\{\left|
  \sum_{i=1}^{n}\varphi(x_{i})\psi(y_{i})\right|\,\, :\,\,
 \varphi\in \mcal{B}_{1}(\mcal{X}^*),\, \psi\in \mcal{B}_{1}(\mcal{Y}^*)
  \right\}.
$$
The completion of $\mcal{X}\otimes \mcal{Y}$ with respect to the injective norm is denoted by $\mcal{X} \otimes^{\varepsilon} \mcal{Y}$. The Banach space $\mcal{X} \otimes^{\varepsilon} \mcal{Y}$ is known as the injective tensor product of the Banach spaces $\mcal{X}$ and $ \mcal{Y}$. 
\end{definition}

\begin{remark}\label{bounded maps}
Let $\mcal{X},\mcal{Y},\mcal{Z}$ and $\mcal{W}$ be the normed spaces
and $T_{1}:\mcal{X}\rightarrow \mcal{W} $ and $T_{2}:
\mcal{Y}\rightarrow \mcal{Z}$ be bounded linear maps. Then, there
exists a unique bounded map $T_{1}\otimes^{\varepsilon} T_{2}:
\mcal{X} \otimes^{\varepsilon} \mcal{Y} \rightarrow \mcal{W}
\otimes^{\varepsilon} \mcal{Z} $ such that $T_{1}\otimes^{\varepsilon}
T_{2}(x\otimes y)= T_{1}(x)\otimes T_{2}(y)$ for every $x\in \mcal{X},
y\in \mcal{Y}$. Also, $\|T_{1}\otimes^{\varepsilon} T_{2}\|=
\|T_{1}\|\|T_{2}\|$.\hfill  (See \cite[Proposition 3.2]{R}, for a proof.)
\end{remark}

The following identification is well known - see, for instance,
(\cite[Lemma 6.4.16]{JM} and \cite[Section 3.2]{R}).

\begin{lemma}\label{identification}
Let $\mcal{X}$ be a Banach space and $C_{0}(\Omega)$ denote the Banach
space of complex continuous functions vanishing at infinity on a
locally compact Hausdorff space $\Omega$.  Then,
$C_{0}(\Omega)\otimes^{\varepsilon} \mcal{X}$ is isometrically
isomorphic to the Banach space $C_{0}(\Omega,\mcal{X})$ of continuous
functions from $\Omega$ into $\mcal{X}$ that vanish at infinity.
\end{lemma}

It is well known that the injective tensor product preserves subspace
in the following sense (see \cite{R} for a proof).
\begin{lemma}
Let $\mcal{X}$ and $\mcal{Y}$ be Banach spaces and $\mcal{Z}$ be a
closed subspace of $\mcal{X}$. Then, the identity map on $\mcal{Z}
\otimes \mcal{Y}$ extends to an isometric map from $\mcal{Z}
\otimes^{\varepsilon}\mcal{Y}$ onto the closed subspace
$\overline{\mcal{Z} \otimes \mcal{Y}}^{\|\cdot\|_{\varepsilon}}$ of
$\mcal{X}\otimes^{\varepsilon} \mcal{Y}$.
\end{lemma}

We now proceed to establish the stability results that we asserted in
the Introduction. To begin with, we have the following analogue of
\cite[Theorem 4.6]{GK2}, whose proof follows verbatim (on the lines of
the proof given for \cite[Theorem 4.6]{GK2}).

\begin{theorem}\label{reverse inequality}
Let $\mcal{X}$ be a Banach space and $C_{0}(\Omega)$ denote the Banach
space consisting of continuous functions vanishing at infinity on a
locally compact Hausdorff space $\Omega$. Then, for any two closed
subspaces $\mcal{Y}$ and $\mcal{Z}$ of $\mcal{X}$,
\[
d_{KK}\left(\mcal{Y}\otimes^{\varepsilon} C_{0}(\Omega), \mcal{Z}\otimes^{\varepsilon} C_{0}(\Omega)\right)\leq d_{KK}(\mcal{Y}, \mcal{Z})
\]
and      
\[
 d_{0}\left(\mcal{Y}\otimes^{\varepsilon} C_{0}(\Omega), \mcal{Z}\otimes^{\varepsilon} C_{0}(\Omega)\right)\leq d_{0}(\mcal{Y}, \mcal{Z}).
\]
\end{theorem}

For the Banach space $C(K)$ of continuous functions on a compact Hausdorff space
 $K$, we have stability for both Kadison-Kastler and Christensen distances.

\begin{theorem}\label{equality}
Let $\mcal{X}$ be a Banach space and $C(K)$ denote the Banach space
consisting of continuous functions on a compact Hausdorff space
$K$. Then, for any two closed subspaces $\mcal{Y}$ and $\mcal{Z}$ of
$\mcal{X}$,
\[
d_{KK}(\mcal{Y}, \mcal{Z})= d_{KK}(\mcal{Y}\otimes^{\varepsilon} C(K),
\mcal{Z}\otimes^{\varepsilon} C(K))
\]
and      
\[
 d_{0}(\mcal{Y}, \mcal{Z})= d_{0}(\mcal{Y}\otimes^{\varepsilon} C(K),
 \mcal{Z}\otimes^{\varepsilon} C(K).
\]

\end{theorem}

\begin{proof}
  We give a proof only for the Christensen distance as the proof for
  the Kadison-Kastler distance is similar.

  In view of \Cref{reverse
    inequality}, it just remains to show that
\[
d_{0}(\mcal{Y}, \mcal{Z})\leq d_{0}(\mcal{Y}\otimes^{\varepsilon} C(K), \mcal{Z}\otimes^{\varepsilon} C(K)). 
\]

First note that the natural isometry $\theta$ from $\mcal{X}$ onto
$\mcal{X}\otimes^{\varepsilon}\mathbb{C}\mathbbm{1}$ maps $\mcal{Y}$
and $\mcal{Z}$ onto
$\mcal{Y}\otimes^{\varepsilon}\mathbb{C}\mathbbm{1}$ and
$\mcal{Z}\otimes^{\varepsilon}\mathbb{C}\mathbbm{1}$ respectively,
where $\mathbbm{1}$ is the constant function taking the value $1$ on
$K$. 

Let $\epsilon>0$
and fix a $\gamma_{0}>0$ such that
\[
d_{0}(\mcal{Y} \otimes^{\varepsilon} C(K), \mcal{Z}\otimes^{\varepsilon} C(K))<  \gamma_{0} < d_{0}(\mcal{Y} \otimes^{\varepsilon} C(K), \mcal{Z}\otimes^{\varepsilon} C(K)) + \epsilon.
\]
This implies that $\mcal{Y} \otimes^{\varepsilon}
C(K)\subseteq_{\gamma_{0}}\mcal{Z}\otimes^{\varepsilon} C(K)$ and that
$\mcal{Z} \otimes^{\varepsilon}
C(K)\subseteq_{\gamma_{0}}\mcal{Y}\otimes^{\varepsilon} C(K)$.
 
Let $y \in B_{1}(\mcal{Y})$. Then, $\theta(y)= y\otimes \mathbbm{1}
\in B_{1}(\mcal{Y} \otimes^{\varepsilon}
\mathbb{C}\mathbbm{1})\subseteq B_{1}(\mcal{Y} \otimes^{\varepsilon}
C(K))$; so, there exists a $z\in \mcal{Z} \otimes^{\varepsilon} C(K)$
such that $\|y\otimes \mathbbm{1}-z\|_{\varepsilon}\leq \gamma_{0}$.
Fix a state $\phi$ on $C(K)$ and, by \Cref{bounded maps}, consider the
natural map
$$ \mathrm{id}_{\mcal{X}}\otimes^{\varepsilon} \phi \mathbbm{1}:
\mcal{X} \otimes^{\varepsilon} C(K) \to \mcal{X} \otimes^{\varepsilon}
\mathbb{C}\mathbbm{1},
$$ with $\|\mathrm{id}_{\mcal{X}}\otimes^{\varepsilon} \phi\|=
\|\mathrm{id}_{\mcal{X}}\|\| \phi \|= 1$.  Clearly, it maps
$\mcal{Z}\otimes^{\varepsilon} C(K)$ onto
$\mcal{Z}\otimes^{\varepsilon} \mathbb{C}\mathbbm{1}$. Thus, $z_{0}:=
(\mathrm{id}_{\mcal{X}}\otimes^{\varepsilon} \phi \mathbbm{1})(z)\in
\mcal{Z} \otimes^{\varepsilon} \mathbb{C}\mathbbm{1}$; so that $\theta^{-1}(z_0) \in \mcal{Z}$  and
\begin{eqnarray*}
\|y-\theta^{-1}(z_{0})\| &=& \|\theta^{-1}(y\otimes \mathbbm{1})-\theta^{-1}(z_{0})\| \\
&=& \|y\otimes \mathbbm{1}- z_{0}\|_{\varepsilon}\\
&=&\|(\mathrm{id}_{\mcal{X}}\otimes^{\varepsilon} \phi \mathbbm{1})(y\otimes \mathbbm{1}-z)\|_{\varepsilon}\\
&\leq & \|y\otimes \mathbbm{1}-z\|_{\varepsilon}\\
& \leq & \gamma_{0}.
\end{eqnarray*}
So, $\mcal{Y} \subseteq_{\gamma_{0}}\mcal{Z}$.
Similarly,  $\mcal{Z}\subseteq_{\gamma_{0}}\mcal{Y}$; 
so that 
\[
d_{0}(\mcal{Y}, \mcal{Z})\leq \gamma_{0}< d_{0}(\mcal{Y} \otimes^{\varepsilon} C(K), \mcal{Z}\otimes^{\varepsilon} C(K))+ \epsilon.
\]
Since $\epsilon>0$ was arbitrary, it follows that
\(
d_{0}(\mcal{Y}, \mcal{Z})\leq  d_{0}(\mcal{Y} \otimes^{\varepsilon} C(K), \mcal{Z}\otimes^{\varepsilon} C(K)).
\) 
\end{proof}

\section{Distance between subalgebras of Banach space projective tensor product of $C^*$-algebras}\label{projective}
\subsection{Banach space projective tensor product}
\begin{definition}
Let $\mcal{X}$ and $\mcal{Y}$ be Banach spaces and $\mcal{X}\otimes \mcal{Y}$ denote their algebraic tensor product. For any element 
$u \in \mcal{X}\otimes \mcal{Y}$, its projective norm is given by 
$$
\|u\|_{\gamma} = \inf \left\{
  \sum_{i=1}^{n}\|x_{i}\|\|y_{i}\|\,\, :\,\,
 u= \sum_{i=1}^{n} x_{i}\otimes y_{i}
  \right\}.
$$ The completion of $\mcal{X}\otimes \mcal{Y}$ with respect to the
  projective norm is denoted by $\mcal{X} \otimes^{\gamma} \mcal{Y}$
  and is known as the projective tensor product of the Banach spaces
  $\mcal{X}$ and $\mcal{Y}$.
\end{definition}

\begin{remark}
  Let $\mC$ and $\mD$ be two $C^*$-algebras. Then, it is easily seen (and well-known) that their projective
  tensor product $\mC \otimes^{\gamma}\mD$ naturally inherits the
  structure of a Banach $*$-algebra.  
\end{remark}

The following remark is similar to \Cref{bounded maps} - see
\cite[Proposition 2.3]{R} for a proof.
\begin{remark}\label{P bounded maps}
Let $\mcal{X},\mcal{Y},\mcal{Z}$ and $\mcal{W}$ be the normed spaces
and $T_{1}:\mcal{X}\rightarrow \mcal{W} $ and $T_{2}:
\mcal{Y}\rightarrow \mcal{Z}$ be bounded linear maps. Then, there
exists a unique bounded linear map $T_{1}\otimes^{\gamma} T_{2}:
\mcal{X} \otimes^{\gamma} \mcal{Y} \rightarrow \mcal{W}
\otimes^{\gamma} \mcal{Z} $ such that $(T_{1}\otimes^{\gamma}
T_{2})(x\otimes y)= T_{1}(x)\otimes T_{2}(y)$ for every $x\in
\mcal{X}, y\in \mcal{Y}$. Also, $\|T_{1}\otimes^{\gamma} T_{2}\|=
\|T_{1}\|\|T_{2}\|$.
\end{remark}

The following well-known result proves to be a very important tool to
compare the distance between subalgebras of a normed algebra and the
distance between their tensor products with a fixed normed algebra -
see \cite[Proposition 2.2]{R} for a proof.
\begin{proposition}{\label{unit-ball}}
Let $\mcal{X}$ and $\mcal{Y}$ be Banach spaces.  Then the closed unit
ball $B_{1}(\mcal{X} \otimes^{\gamma} \mcal{Y})$ of $\mcal{X}
\otimes^{\gamma} \mcal{Y}$ is the closed convex hull of the set
$B_{1}(\mcal{X})\otimes B_{1}(\mcal{Y})$.
\end{proposition}

In general, given a subspace $\mcal{Z}$ of $\mcal{X}$, the projective
tensor product $\mathcal{Z} \otimes^{\gamma} \mathcal{Y}$ cannot be
identified isometrically with a subspace of $\mathcal{X}
\otimes^{\gamma} \mathcal{Y}$.  However, if we start with
$C^*$-algebras, we have the following positive result.

\begin{proposition}\cite[Theorem 2.6]{GR}\label{injectivity}
Let $\mC$ and $\mD$ be two $C^*$-algebras and $\mA$ be a
$C^*$-subalgebra of $\mC$. Then the identity map on $\mA \otimes \mD$
extends to an isometric $*$-algebra map from $\mA \otimes^{\gamma}\mD$
onto the closed $*$-subalgebra $\overline{\mA \otimes
  \mD}^{\|\cdot\|_{\gamma}}$ of $\mC\otimes^{\gamma} \mD$.
\end{proposition}

The following proposition is the natural analogue of \cite[Proposition
  4.1]{GK2}.
\begin{proposition}
Let $\mC$ and $\mD$ be $C^*$-algebras. Let $\mA$ and $\mB$ be
$C^*$-subalgebras of $\mC$, and let $\mP$ be a $C^*$-subalgebra of
$\mD$. If there exists a conditional expectation from $\mD$ onto
$\mP$, then
\[
d_{0}(\mA \otimes^{\gamma} \mP, \mB\otimes^{\gamma} \mP)\leq d_{0}(\mA
\otimes^{\gamma} \mD, \mB\otimes^{\gamma} \mD)
\]
and 
\[
d_{KK}(\mA \otimes^{\gamma} \mP, \mB\otimes^{\gamma} \mP)\leq
d_{KK}(\mA \otimes^{\gamma} \mD, \mB\otimes^{\gamma} \mD).
\]
\end{proposition}
\begin{proof}
Let $E: \mD \to \mP$ be a conditional expectation. By
\Cref{injectivity}, we can identify $\mA \otimes^{\gamma} \mP$ and
$\mB \otimes^{\gamma} \mP$ with Banach $*$-subalgebras of $\mC
\otimes^{\gamma} \mD$.  Let $\epsilon>0$ and fix a $\gamma_{0}>0$ such
that
\[
d_{0}(\mA \otimes^{\gamma} \mD, \mB\otimes^{\gamma} \mD)< \gamma_{0} <
d_{0}(\mA \otimes^{\gamma} \mD, \mB\otimes^{\gamma} \mD) + \epsilon.
\]
This implies that $\mA \otimes^{\gamma}
\mD\subseteq_{\gamma_{0}}\mB\otimes^{\gamma} \mD$ and $\mB
\otimes^{\gamma} \mD\subseteq_{\gamma_{0}}\mA\otimes^{\gamma} \mD$.

Let $x\in B_{1}(\mA \otimes^{\gamma} \mP)\subseteq B_{1}(\mA
\otimes^{\gamma} \mD)$. Then, there exists a $y\in \mB
\otimes^{\gamma} \mD$ such that $\|x-y\|_{\gamma}\leq
\gamma_{0}$. Consider the map $\mathrm{id}_{\mC}\otimes^{\gamma} E:
\mC \otimes^{\gamma} \mD \to \mC \otimes^{\gamma} \mP$ by \Cref{P
  bounded maps} with $\|\mathrm{id}_{\mC}\otimes^{\gamma} E\|=
\|E\|\leq 1$. Clearly, it fixes $\mC\otimes^{\gamma} \mP$. Thus,
$y_{0}:= (\mathrm{id}_{\mC}\otimes^{\gamma} E)(y)\in \mB
\otimes^{\gamma} \mP$ and
\begin{eqnarray*}
\|x-y_{0}\|_{\gamma} &=& \|(\mathrm{id}_{\mC}\otimes^{\gamma} E)(x-y)\|_{\gamma}\\
&\leq & \|x-y\|_{\gamma}\leq \gamma_{0}.
\end{eqnarray*}
So, $\mA \otimes^{\gamma} \mP\subseteq_{\gamma_{0}}\mB \otimes^{\gamma}
\mP$. Similarly,  $\mB \otimes^{\gamma} \mP\subseteq_{\gamma_{0}}\mA \otimes^{\gamma}
\mP$. Thus, 
\[
d_{0}(\mA \otimes^{\gamma}\mP, \mB \otimes^{\gamma} \mP)\leq \gamma_{0}< d_{0}(\mA \otimes^{\gamma}
\mD, \mB \otimes^{\gamma} \mD)+ \epsilon.
\]
Hence,
\[
d_{0}(\mA \otimes^{\gamma} \mP, \mB \otimes^{\gamma} \mP) \leq
d_{0}(\mA \otimes^{\gamma} \mD, \mB \otimes^{\gamma} \mD).
\] 
The proof for the Kadison-Kastler distance is analogous and we leave
the details to the reader.
\end{proof}

The following is an easy consequence of the preceding proposition. 
\begin{cor}\label{D-unital}
 Let $\mC $ and $\mD $ be $C^*$-algebras and, $\mA$ and $\mB$ be
 $C^*$-subalgebras of $\mC$. If $\mD$ is unital, then
\[
d_{KK}(\mA, \mB)\leq d_{KK}(\mA \otimes^{\gamma} \mD, \mB \otimes^{\gamma} \mD)
\]
 and \[
d_{0}(\mA, \mB)\leq d_{0}(\mA \otimes^{\gamma} \mD, \mB \otimes^{\gamma} \mD).
\]
\end{cor}

\begin{theorem}\label{Projective-Ch}
Let $\mC$ and $\mD$ be $C^*$-algebras and, $\mA$ and $\mB$ be
$C^*$-subalgebras of $\mC$. If $\beta$ is a positive scalar such that
$\mA \subseteq_{\beta}\mB$, then $\mA \otimes^{\gamma} \mD
\subseteq_{\beta}\mB \otimes^{\gamma} \mD$.  In particular,
$$d_0(\mA \otimes^{\gamma} \mD, \mB \otimes^{\gamma} \mD) \leq
d_0(\mA, \mB).$$

Moreover, if $\mD$ is unital, then
$$d_{0}(\mA \otimes^{\gamma} \mD, \mB \otimes^{\gamma} \mD) = d_{0}(\mA, \mB).$$ 
\end{theorem}

\begin{proof}
As $B_{1}(\mA \otimes \mD)$ is dense in $B_{1}(\mA \otimes^{\gamma}
\mD)$, it is enough to show that $\mA \otimes \mD \subseteq_{\beta}\mB
\otimes \mD$.

Let $x\in B_{1}(\mA \otimes \mD)$. Then, by \Cref{unit-ball} $x=
\sum_{i=1}^{n}\alpha_{i} (a_{i}\otimes d_{i})$, where $a_{i}\in
B_{1}(\mA)$, $d_{i}\in B_{1}(\mD)$, $0\leq \alpha_{i}\leq 1$ and
$\sum_{i=1}^{n}\alpha_{i}= 1$. As $\mA \subseteq_{\beta} \mB$, for
each $a_{i}$, there exists a $b_{i}$ in $\mB$ such that
\(
\|a_{i}-b_{i}\| \leq  \beta.
\) 
Thus, $y := \sum_{i=1}^{n}\alpha_{i} (b_{i}\otimes d_{i})\in \mB\otimes \mD$ and 
\[
\|x-y\|_{\gamma} = 
\|\sum_{i=1}^{n}\alpha_{i} (a_{i}-b_{i})\otimes d_{i}\|_{\gamma} \leq  \sum_{i=1}^{n}\alpha_{i} \|a_{i}-b_{i}\|\|d_{i}\|  \leq
  \beta.
\]
Hence, $\mA \otimes \mD \subseteq_{\beta}\mB
\otimes \mD$.

When $\mcal{D}$ is unital, then the final equality follows from
\Cref{D-unital}.
\end{proof}
On similar lines, we also obtain the following:
\begin{theorem}\label{Projective-KK}
Let $\mC$ and $\mD$ be $C^*$-algebras and, $\mA$ and $\mB$ be
 $C^*$-subalgebras of $\mC$. Then,
$$ d_{KK}(\mA \otimes^{\gamma} \mD, \mB \otimes^{\gamma} \mD) \leq
d_{KK}(\mA, \mB).$$ Moreover, if $\mD$ is unital, then
$$d_{KK}(\mA \otimes^{\gamma} \mD, \mB \otimes^{\gamma} \mD) = d_{KK}(\mA, \mB).$$  
\end{theorem}

\section{Some concrete calculations}
In \cite{GK}, some concrete calculations were made regarding
the Kadison-Kastler and Christensen distances between crossed-product
subalgebras. Interestingly, we observed that when a discrete group $G$
acts on a $C^*$-algebra $\mA$ via a group homomorphism $\alpha: G
\rightarrow \Aut(\mA) $, then for any pair of distinct subgroups $H$
and $K$ of $G$, 
\begin{enumerate}
\item $d_{KK}(A\rtimes^r_{\alpha}H, A\rtimes^r_{\alpha}K)=1 =
  d_{0}(A\rtimes^r_{\alpha}H, A\rtimes^r_{\alpha}K)$ in $\mA \rtimes^r_{\alpha}
  G$; and,
\item $d_{KK}(A\rtimes^u_{\alpha}H, A\rtimes^u_{\alpha}K)=1 =
  d_{0}(A\rtimes^u_{\alpha}H, A\rtimes^u_{\alpha}K) $ in
  $A\rtimes^u_{\alpha}G$.
\end{enumerate} 
In this section, we make analogues calculations for reduced twisted
crossed-product subalgebras in $C^*$-algebras and for crossed-product
subalgebras in von Neumann algebras.

\subsection{Distance between subalgebras of reduced twisted
  crossed product $C^*$-algebras}\label{twisted} 

\subsubsection{Reduced twisted crossed product}
Let $G$ be a discrete group and $\mA$ be a unital $C^*$-algebra. A
discrete twisted $C^*$-dynamical system is a quadruple
$(\mA,G,\alpha,\sigma)$ of a unital $C^*$-algebra $\mA$, a discrete
group $G$ and a pair of maps $\alpha : G \rightarrow
\Aut(\mA)$ and $\sigma: G\times G \rightarrow \mathcal{U}(\mA)$
satisfying
\begin{eqnarray*}
\alpha_{s}\circ \alpha_{t} & = & Ad_{\sigma(s,t)}\circ
\alpha_{st};\\ \sigma(r,s)\sigma(rs,t) & = &
\alpha_{r}(\sigma(s,t))\sigma(r,st);\, \text{and},\\ \sigma(s,e) & = &
\sigma(e,s)=1
\end{eqnarray*}
for all $r,s,t\in G$.

For a given discrete twisted $C^*$-dynamical system
$(\mA,G,\alpha,\sigma)$, where $\mA$ is faithfully and
non-degenerately represented on a Hilbert space $\mcal{H}$, the
reduced twisted crossed product $\mA \rtimes^{r}_{(\alpha,\sigma)} G$
(sometimes also denoted by $C^{*}_{r}(\mA,G,\alpha, \sigma)$) can be
defined as the $C^*$-algebra generated by $\pi_{\alpha}(\mA)$ and
$\lambda_{\sigma}^G(G)$ inside $\mcal{B}(l^2(G,\mcal{H}))$, where
$\pi_{\alpha}$ is the faithful representation of $\mA$ on
$l^2(G,\mcal{H})$ defined by
\[
(\pi_{\alpha}(a)\zeta)(h)=\alpha_{h^{-1}}(a)\zeta(h),
\]
and, for each $g\in G$, $\lambda^{G}_{\sigma}(g)$ is the unitary
operator on $l^2(G,\mcal{H})$ defined by
\[
(\lambda^{G}_{\sigma}(g)\zeta)(h)= \sigma(h^{-1},g)\zeta(g^{-1}h),
\]
for $ a\in \mA, \zeta\in l^2(G,\mcal{H}), h\in G$. The pair
$(\pi_{\alpha},\lambda^{G}_{\sigma})$ satisfies the
covariant relations
\begin{eqnarray*}
\pi_{\alpha}(\alpha_{g}(a)) & = &
\lambda^{G}_{\sigma}(g)\pi_{\alpha}(a)\lambda^{G}_{\sigma}(g)^{*}
\text{\, and }\\ \lambda^{G}_{\sigma}(g)\lambda^{G}_{\sigma}(h) & = &
\pi_{\alpha}(\sigma(g,h))\lambda^{G}_{\sigma}(gh)
\end{eqnarray*}
for all $a\in \mA$ and $g,h\in G$. It is well known that, upto
$*$-isomorphism, $\mA \rtimes^{r}_{(\alpha,\sigma)} G$ is independent
of the faithful representation of $\mA$. For convenience, we shall
identify $\mA$ with $\pi_{\alpha}(\mA)$. Consider the subset
\[
C_{c}(G,\mA,\alpha,\sigma):=
\left\{\sum_{g \in F}a_{g}\lambda^{G}_{\sigma}(g)\,\,: \ F \in \mcal{F}(G),\,\,a_{g}\in
\mA, \,\, g\in G\right\}
\]
of $\mA \rtimes^{r}_{(\alpha,\sigma)}G$, where $\mcal{F}(G)$ denotes
the collection of finite subsets of $G$. Then,
$C_{c}(G,\mA,\alpha,\sigma)$ is a unital dense $*$-subalgebra of $\mA
\rtimes^{r}_{(\alpha,\sigma)}G$. Additionally, we have
\begin{eqnarray*}
(\lambda^{G}_{\sigma}(t))^* &=&
  \lambda^{G}_{\sigma}(t^{-1})\sigma(t,t^{-1})^*=
  \sigma(t,t^{-1})^*\lambda^{G}_{\sigma}(t^{-1}),\,\,\text{and}\\ x^{*}_{g}
  &=& \alpha_{g}(x_{g^{-1}})^{*}\sigma(g^{-1},g)^{*}\\ &=&
  \sigma(g,g^{-1})^{*} \alpha_{g}(x_{g^{-1}})^{*}
\end{eqnarray*}
 for all $t,\,g\in G.$

The following observations in the reduced twisted crossed product are
very useful - for more, see \cite{BG}, \cite{B}.
   
\begin{remark}\label{r-eta-relation}
Let $\mA, G$ and $\alpha$ be as above and $H$ be a subgroup of $G$. 
\begin{enumerate}

\item There exists a faithful conditional expectation $E:\mA
  \rtimes^{r}_{(\alpha,\sigma)}G \rightarrow \mA $ such that
  $E(\lambda^{G}_{\sigma}(t))=0$ for all $t\in G$, $t \neq
  e$. (\cite[Theorem 2.2]{B})

\item Consider $\mA \rtimes^{r}_{(\alpha,\sigma)}G$ as a (right) pre-Hilbert
$C^*$-module (over $\mA$) with respect to the $\mA$-valued inner product
  given by
\[
\langle x,y\rangle_{\mA}=E(x^*y),\,\, \text{for}\,\, x,y\in \mA \rtimes^{r}_{(\alpha,\sigma)}G.
\]
Further, consider the identity map $\eta_{E}: \mA
\rtimes^{r}_{(\alpha,\sigma)}G \rightarrow \mA
\rtimes^{r}_{(\alpha,\sigma)}G$, where the co-domain is treated as the
pre-Hilbert $\mA$-module. Since $E$ is a contraction,
\[
\|\eta_{E}(x)\| =\|E(x^*x)\|^{\frac{1}{2}}\leq \|x\|
\]
for all $x\in \mA \rtimes^{r}_{(\alpha,\sigma)}G$.

\item The injective map $C_{c}(H,\mA,\alpha,\sigma) \ni z \mapsto z\in
  C_{c}(G,\mA,\alpha,\sigma)$ extends to an injective $*$-homomorphism
  from $\mA \rtimes^{r}_{(\alpha,\sigma)}H$ into $\mA
  \rtimes^{r}_{(\alpha,\sigma)}G$. Thus, $\mA
  \rtimes^{r}_{(\alpha,\sigma)}H$ can be considered as a
  $C^*$-subalgebra of $\mA \rtimes^{r}_{(\alpha,\sigma)}G$.\hfill (See
  \cite[$\S 4.26$]{Z-M}.)
\end{enumerate}
\end{remark}
The following useful observation follows from \cite[Corollary
  2.3]{B}. However, we provide a direct proof for the sake of
completeness.
\begin{lemma}\label{xg-inequality}
Let $(\mA,G,\alpha, \sigma)$ be a (discrete) twisted $C^*$-dynamical
system. If $x= \sum_{g\in F}x_{g}\lambda_{\sigma}(g)\in
C_{c}(G,\mA,\sigma)$ for some $F\in \mathcal{F}(G)$, then
\[
\|x_{g}\|^2\leq \|E(x^*x)\|=\|\eta_{E}(x)\|^{2}
\]
for all $g\in F$.
\end{lemma}
\begin{proof}
Notice that, for any $g\in G$,
\[
(x^*x)_g = \sum_{h\in G}x^{*}_{h}\alpha_{h}(x_{h^{-1}g})\sigma(h, h^{-1}g),
\]
with the convention that $x_t=0$ for $t \notin F$. Thus,
\begin{eqnarray*}
E(x^*x)&=& E(\sum_{g\in G}x^*x(g)\lambda^{G}_{\sigma}(g))=(x^*x)_e\\
&=& \sum_{h\in G}x^{*}_{h}\alpha_{h}(x_{h^{-1}})\sigma(h, h^{-1})\\
&=& \sum_{h\in G}\sigma(h, h^{-1})^{*}\alpha_{h}(x^{*}_{h^{-1}}x_{h^{-1}})\sigma(h,h^{-1})\\
&=&\sum_{h\in G}\sigma(h, h^{-1})^{*}(\alpha_{h}\circ\alpha_{h^{-1}})(\alpha_{h}(x^{*}_{h^{-1}}x_{h^{-1}}))\sigma(h,h^{-1})\\
&=& \sum_{h\in G} \alpha_{h}(x^{*}_{h^{-1}}x_{h^{-1}})\\
&=& \sum_{h\in G} \alpha_{h^{-1}}(x^{*}_{h}x_{h}).
\end{eqnarray*} 
In particular, $E(x^*x)\geq \alpha_{g^{-1}}(x_{g}^*x_{g})$ for every $g\in F$. Thus,
\[
\|\eta_{E}(x)\|^{2}= \|E(x^*x)\|\geq \|\alpha_{g^{-1}}(x_{g}^*x_{g})\|=\|x_{g}\|^{2}
\]
for all $g\in F$.
\end{proof}

\begin{proposition}\label{sg-distance}
Let $(\mA,\,G,\,\alpha,\,\sigma)$ be a (discrete) twisted
$C^*$-dynamical system and let $H$ and $K$ be two distinct
subgroups of $G$. Then,
\[
d_{KK}(C_{c}(H,\mA,\alpha, \sigma), C_{c}(K,\mA,\alpha, \sigma))= 1=
d_{0}(C_{c}(H,\mA,\alpha, \sigma), C_{c}(K,\mA,\alpha,\sigma))
\]
in $\mA \rtimes^{r}_{(\alpha,\sigma)}G$.
\end{proposition}
\begin{proof}
Note that, $d_{KK}(C_{c}(H,\mA,\sigma), C_{c}(K,\mA,\sigma))\leq 1$,
by \Cref{KK-facts}(1), and
\begin{eqnarray*}
d_{0}(C_{c}(H,\mA,\alpha,\sigma), C_{c}(K,\mA,\alpha, \sigma))\leq
d_{KK}(C_{c}(H,\mA,\alpha, \sigma), C_{c}(K,\mA,\alpha, \sigma)),
\end{eqnarray*}
 by  \Cref{d0-basics}. So, it just remains  to show that
\[ d_{0}(C_c(H, \mA, \sigma),C_c(K, \mA, \sigma)) \geq 1.
\]
Since $H$ and $K$ are distinct, either $H \neq H \cap K$
or $K \neq H \cap K$.  Without loss of generality, we can assume that
$H\neq H\cap K$. Then, in view of
\Cref{r-eta-relation}(2) and \Cref{xg-inequality}, we observe that
\[
\|\lambda_{\sigma}(h) - x\|\geq \|\eta_{E}(\lambda_{\sigma}(h) -
x)\|_{\mA} \geq 1 
\]
for all $h \in H \setminus H \cap K$, $x \in C_c(K, \mA,\alpha,
\sigma)$.
 This shows that $C_c(H, \mA,\alpha,
\sigma) \nsubseteq_{\beta} C_c(K, \mA,\alpha, \sigma)$ whenever $0< \beta < 1$.
In other words, if  $C_c(H, \mA,\alpha,
\sigma) \subseteq_{\beta} C_c(K, \mA,\alpha, \sigma)$ for some $\beta
> 0$, then $\beta$ must be $ \geq 1$.

So, by the definition of $d_0$, we must
have
  \[
  d_{0}(C_c(H, \mA,\alpha, \sigma), C_c(K, \mA,\alpha,\sigma)) \geq 1.
  \]
\end{proof}
In view of \Cref{KK-facts}, \Cref{d0-basics} and the preceding
proposition, we readily deduce the following:
\begin{cor}
Let $G, H, K, \mA,\alpha$ and $\sigma$ be as in \Cref{sg-distance}. Then,
\[
d_{KK}(\mA \rtimes^{r}_{(\alpha,\sigma)}H, \mA \rtimes^{r}_{(\alpha,\sigma)}K)=1= d_{0}(\mA \rtimes^{r}_{(\alpha,\sigma)}H, \mA \rtimes^{r}_{(\alpha,\sigma)}K).
\]
\end{cor} 

\subsection{Distance between subalgebras of crossed-product von Neumann algebras}\label{crossedv}
Let $G$ be a discrete group acting on a von Neumann algebra $\mcal{M}
\subseteq B(\mcal{H}) $ by the $*$-automorphisms $\alpha_{g}, g\in
G$. Let $\mcal{\tilde{H}}= l^{2}(G,\mcal{H})$ be the set of all square
summable $\mcal{H}$-valued functions on $G$. Consider the faithful
representations $\pi$ of $\mcal{M}$ and $\lambda$ of $G$ on
$\mcal{\tilde{H}}$ given by
\[
\pi(a)\xi(g)= \alpha_{g^{-1}}(a)\xi(g),\,\,\,\,\,
\lambda(g)\xi(h)= \xi(g^{-1}h),
\]
for all $\xi \in \mcal{\tilde{H}}, g\in G, h\in G, a\in
\mcal{M}$. These representations satisfy the covariance relation
\[
\lambda(g)\pi(a)\lambda(g)^{*}= \pi(\alpha_{g}(a))
\]
for all $g\in G$ and $a\in \mcal{M}$. The von Neumann algebra
generated by $\pi(\mcal{M})$ and $\lambda(G)$ is known as the
crossed-product algebra of $\mcal{M}$ by $G$ with respect to $\alpha$,
and it is denoted by $\mcal{M} \rtimes_{\alpha}G$.

The following observations are very useful in the crossed-product von
Neumann algebras. For  details, see   \cite[Page 255]{M}.
\begin{remark}
Let $\mcal{M}$, $G$ and $\alpha$ be as above.
\begin{enumerate}
\item Consider $\mcal{M}_{0}:= \{\sum_{g\in
  G}\pi(x(g))\lambda(g)\,\,|\,\,x: G\rightarrow \mcal{M}\,\text{is
  finitely supported}\}$. Then, $\mcal{M}_{0}$ is a $\sigma$-weakly
  dense unital $*$-subalgebra of $\mcal{M} \rtimes_{\alpha}G$.
\item
  For each $g\in G $, consider the operator 
\[
P_{g}: \mcal{\tilde{H}}\rightarrow \mcal{H}\,\, \text{given by}\,\,\,
P_{g}(\xi)=\xi(g^{-1}),\,\xi\in \mcal{\tilde{H}}.
\]
These operators satisfy the following properties:
\begin{enumerate}
  \item
    $ P_{g}(\lambda(h))= P_{gh}$;
  \item $P_{g}(\pi(x))=\alpha_{g}(x)P_{g}$;
  \item $P_{g}\pi(x)P_{g}^{*}=\alpha_{g}(x)$; and,
    \item $\sum^{\sigma \text{-}\mathrm{weak}}_{g\in
      G} P_{g}^{*}\alpha_{g}(x)P_{g}= \pi (x),$
      \end{enumerate}
for all $g,h\in G$ and $x\in \mcal{M}$.
\item Define $E: B( \mcal{\tilde{H}})\rightarrow B(\mcal{H})$ by $E(x)= P_{e}xP^{*}_{e}$, $x \in B( \mcal{\tilde{H}})$. Then, $E(\mcal{M} \rtimes_{\alpha}G) = \mcal{M}$ and  
$$
E(\lambda(g)x\lambda(g)^*)= \alpha_{g}(E(x)),
$$
for all $x\in \mcal{M} \rtimes_{\alpha}G, g \in G$. 
\item[(iv)] $\tilde{E}:= \pi \circ E : \mcal{M} \rtimes_{\alpha}G
  \rightarrow \pi(\mcal{M})$ is a faithful normal conditional
  expectation.
\end{enumerate}
\end{remark}
The following elementary lemma was mentioned in \cite{M} and will be useful ahead.
\begin{lemma}\label{relation}
For any $z\in \mcal{M} \rtimes_{\alpha}G$, we have
\[
\tilde{E}(z)= \sum_{g\in G}Q_{g}zQ_{g},
\]
where $Q_{g}= P^{*}_{g}P_{g}$ and the sum converges in the $\sigma$-weak topology. 
\end{lemma}

We now quickly recall the Fourier series expansion of elements of the
crossed-product algebra $\mcal{M}\rtimes_{\alpha}G$ given by
Mercer (\cite{M}). Consider the semi-norms on $\mcal{M}\rtimes_{\alpha}G$ given
by $x\mapsto \omega\circ E(x^*x)^\frac{1}{2}$, for all $\omega \in
\mcal{M}_{*}$. The topology generated by this separating family of
semi-norms is called the Bures topology on $\mcal{M}\rtimes_{\alpha}G$
and it will be denoted by $\tau_{\mathcal{B}}$.

\begin{lemma}\cite{M}
  Let $\mcal{M}, G$ and $\alpha$ be as above. Then, the following hold:
  \begin{enumerate}
\item $x = \sum^{\tau_{\mathcal{B}}}_{g\in G}\pi(x(g))\lambda(g)\,$
  for all $x \in \mcal{M}\rtimes_{\alpha}G$, where
  $x(g):=E(x\lambda(g)^*)$.
  \item For $x,y \in \mcal{M}\rtimes_{\alpha}G$, the multiplication
    and $*$-operations in $\mcal{M}\rtimes_{\alpha}G$ are given by
\begin{eqnarray*}
(xy)(g) & = & \sum^{\sigma\text{-}\mathrm{weak}}_{h\in
    G}x(h)\alpha_{h}(y(h^{-1}g));\,\text{and,}\\ x^{*}(g) & = &
  \alpha_{g}(x(g^{-1})^*),\ \text{for all } g\in G.
\end{eqnarray*} 
\end{enumerate}\end{lemma}

\begin{lemma}\label{inequality}
Let $\mcal{M}, G$ and $\alpha$ be as above. Then, for every $x\in
\mcal{M}\rtimes_{\alpha}G$, 
\[
\|x(g)\|\leq \|E(x^*x)\|^\frac{1}{2}\leq \|x\|
\]
for all $g\in G$.
\end{lemma}

\begin{proof}
Let $x\in \mcal{M}\rtimes_{\alpha}G$. Then, $x=
\sum^{\tau_{\mcal{B}}}_{g\in G}\pi(x(g))\lambda(g) $. Further,
\begin{eqnarray*}
(x^*x)(e)&=& \sum^{\sigma \text{-weak}}_{h\in
  G}x^*(h)\alpha_{h}(x(h^{-1}))= \sum^{\sigma \text{-weak}}_{h\in
  G}\alpha_{h}(x(h^{-1})^*)\alpha_{h}(x(h^{-1}))\\ &=& \sum^{\sigma
  \text{-weak}}_{h\in G} \alpha_{h^{-1}}(x(h)^*x(h)).
\end{eqnarray*}
Thus, we get  
$$
E(x^*x)= (x^*x)(e)= \sum^{\sigma \text{-}\mathrm{weak}}_{h\in G} \alpha_{h^{-1}}(x(h)^*x(h)),
$$
which implies that
\[
E(x^*x)\geq \alpha_{g^{-1}}(x(g)^*x(g))\,\,\forall\,\,g\in G.
\]
Hence, 
$$
\|E(x^*x)\|\geq \|\alpha_{g^{-1}}(x(g)^*x(g))\|= \|x(g)^*x(g)\|= \|x(g)\|^2\,\, \forall\,g\in G.
$$

\end{proof}

\noindent {\bf Notation:} Let $\mcal{M}, G$ and $\alpha$ be as above and $H$ be a subgroup of
$G$. Consider the unital $*$-subalgebra
\[
\mcal{M}\boxtimes H := \left\{\sum_{g\in G}\pi(a(g))\lambda(g)\,\,|\,\,a: G\rightarrow \mcal{M}\,\text{of finite support and } \,
 a(g)=0 \, \forall\,g \in G\setminus H\right\}
\]
of $\mcal{M}\rtimes_{\alpha} G$. Then, the von Neumann subalgebra of
$\mcal{M}\rtimes_{\alpha} G$ generated by $\pi(\mcal{M})$ and
$\lambda(H)$ equals $(\mcal{M}\boxtimes H)^{''}$, where $\pi$ and
$\lambda$ are as above. Further, it is known that $(\mcal{M}\boxtimes
H)^{''}$ is isomorphic to $\mcal{M}\rtimes_{\alpha_{\restriction_H}}
H$ - see \cite{Z-M}. For convenience, we shall write $\mcal{M}
\rtimes_{\alpha} H$ for $(\mcal{M}\boxtimes H)^{''}$.  (Note that
$\mcal{M}\boxtimes H$ is not a standard notation.)

\begin{theorem}
Let $\mcal{M}, G$ and $\alpha$ be as above and, $H$ and $ K$ be two
distinct subgroups of $G$. Then,
\[
d_{KK}(\mcal{M} \rtimes_{\alpha}H, \mcal{M} \rtimes_{\alpha}K)=1=
d_{0}(\mcal{M} \rtimes_{\alpha}H, \mcal{M} \rtimes_{\alpha}K)
\]
in $\mcal{M} \rtimes_{\alpha}G$.
\end{theorem}

\begin{proof}
Without loss of generality, assume that $H\setminus K \neq
\emptyset$. Let $h\in H \setminus K$. Then, $\lambda(h)\in
B_1(\mcal{M} \rtimes_{\alpha}H)$ and for any $y\in \mcal{M}
\rtimes_{\alpha}K$, we get
\begin{eqnarray*}
\|\lambda(h)-y\|^2&=& \|\lambda(h)- \sum^{\tau_{\mcal{B}}}_{k\in
  K}\pi(y(k))\lambda(k)\|^2 \\ &\geq&\| E(z^*z)\| \geq \|z(h)\|^2
\,\,\,\text{(by \Cref{inequality})},
\end{eqnarray*}
where $z= \lambda(h)- \sum^{\tau_{\mcal{B}}}_{k\in K}\pi(y(k))\lambda(k)$. Note that
\begin{eqnarray*}
z(h)= E(z\lambda(h)^*)&=& E((\lambda(h)- \sum^{\tau_{\mcal{B}}}_{k\in K}\pi(y(k))\lambda(k))\lambda(h)^*)\\
&=& E(\lambda(e)- \sum^{\tau_{\mcal{B}}}_{k\in K}\pi(y(k))\lambda(kh^{-1}))\\
&=& P_e(\lambda(e)- \sum^{\tau_{\mcal{B}}}_{k\in K}\pi(y(k))\lambda(kh^{-1}))P^{*}_e\\
&=& P_eP^{*}_e\\
& = & \mathrm{id}_{\mcal{H}}. 
\end{eqnarray*}
This implies that $\|\lambda(h)-y\|^2\geq 1$ for all $h \in H \setminus K$ and $y \in \mcal{M}\rtimes_{\alpha} K$, which shows that $\mcal{M}
\rtimes_{\alpha}H \nsubseteq_{\gamma} \mcal{M} \rtimes_{\alpha}K$ whenever $0 < \gamma < 1$. Thus, if  $\mcal{M}
\rtimes_{\alpha}H \subseteq_{\gamma} \mcal{M} \rtimes_{\alpha}K$, then 
$\gamma$ must be $ \geq 1$, which implies that 
 \[
 d_{0}(\mcal{M} \rtimes_{\alpha}H, \mcal{M} \rtimes_{\alpha}K)\geq 1.
 \]
Hence, by \Cref{d0-basics}(1), (3), it follows that
\[
d_{0}(\mcal{M} \rtimes_{\alpha}H, \mcal{M}
\rtimes_{\alpha}K)=1=d_{KK}(\mcal{M} \rtimes_{\alpha}H, \mcal{M}
\rtimes_{\alpha}K).
\]

\end{proof}

\begin{remark}\label{invariance}
Suppose that $ G$ and $ \alpha$ are as above and $\mcal{M}$ is a
finite von Neumann algebra with a faithul normal $G$-invariant tracial
state $\tau$. Then, it is known that $\mcal{M} \rtimes_{\alpha}G$ is a
finite von Neumann algebra with a faithful normal tracial state
$\tilde{\tau}$ - see \cite[Proposition 1.3.2]{JS}.  Further, by applying
\Cref{relation}, it is seen that
$$\tau \circ
\tilde{E}(z)= \tilde{\tau}(z)
$$ 
for all $z\in \mcal{M} \rtimes_{\alpha}G$.
\end{remark}

\begin{theorem}
Let $\mcal{M}, G$ and $\alpha$ be as in \Cref{invariance}. If $H$ and $
K$ are two distinct subgroups of $G$, then
\[
d_{MT}(\mcal{M} \rtimes_{\alpha}H, \mcal{M} \rtimes_{\alpha}K)=1.
\]
\end{theorem}
\begin{proof}
By \Cref{SOT} (1), it is enough to show that $d_{MT}(\mcal{M}\boxtimes
H, \mcal{M}\boxtimes K)\geq 1$.

As $H$ and $K$ are distinct subgroups of $G$, without loss of
generality, assume that $H\setminus K \neq \emptyset$. Let $h\in H
\setminus K$. Then, $\widehat{\lambda(h)}\in \widehat{B_1(\mcal{M}
  \boxtimes H)}$ and for any ${y}= {\sum_{k\in
    K}\pi(y(k))\lambda(k)} \in B_1(\mcal{M} \boxtimes K)$, we
get
\begin{eqnarray*}
\|\widehat{\lambda(h)}-\widehat{y}\|^{2}_{\tilde{\tau}}&=&
\tilde{\tau}((\lambda(h)-y)^*(\lambda(h)-y))\\ &=& \tau\circ
\tilde{E}((\lambda(h)-y)^*(\lambda(h)-y))\,\,\, \, \,(\text{by
  \Cref{invariance}})\\ &\geq&\tau\circ\pi(\alpha_{h^{-1}}((z(h)^*z(h)))\,\,\,
\, \,(\text{see the proof of \Cref{inequality}})\\ &=& \tau \circ \pi
(\alpha_{h^{-1}}(\mathrm{id}_{\mcal{H}}))= 1,
\end{eqnarray*}
where $z= \lambda(h)-\sum_{k\in K}\pi(y(k))\lambda(k)$.  This implies
that $\|\widehat{\lambda(h)}-\widehat{y}\|^{2}_{\tilde{\tau}}\geq 1$
for all $\widehat{y}\in \widehat{B_1(\mcal{M} \boxtimes K})$. Hence,
\[
\sup_{\widehat{x}\in \widehat{B_1(\mcal{M} \boxtimes
    H)}}d(\widehat{x}, \widehat{B_1(\mcal{M} \boxtimes K}))\geq
d(\widehat{\lambda(h)}, \widehat{B_1(\mcal{M} \boxtimes K})\geq 1.
\]
Thus, $d_{MT}(\mcal{M} \boxtimes H, \mcal{M} \boxtimes K)\geq 1$, and
we are done.
\end{proof}

\noindent{{\bf Acknowledgements:} I would like to sincerely thank my
  supervisor, Dr.~Ved Prakash Gupta, for his valuable feedback and
  suggestions on this article.  }

\end{document}